\documentclass[11pt]{amsart}
\usepackage{amssymb,amsfonts,amscd, color}
\usepackage{amsfonts,amssymb,amscd,subfigure}
\usepackage{amsmath}
\usepackage{srcltx} 
\usepackage{latexsym} 
\usepackage[titletoc]{appendix}
\usepackage{esint}
\usepackage{multicol}
\usepackage{lineno}
\usepackage{mathrsfs}
\usepackage{algorithm,colonequals}
\usepackage{bm}
\usepackage{enumerate}
\usepackage{mathrsfs}
\usepackage{mathtools}
\usepackage{wasysym}
\usepackage{fancyhdr}
\usepackage[plainpages=true,pdfpagelabels,hypertexnames=true,colorlinks=true,pdfstartview=FitV,linkcolor=blue,citecolor=blue,urlcolor=green]{hyperref}

\def\nc{\newcommand}
\def\ga{\gamma}\def\de{\delta}

\def\ep{\epsilon}
\def\bff{{\bf f}}

\def\om{\omega}

 \def\Om{\Omega}

\def\la{\lambda}
\def\La{\Lambda}

\def\integral{\int}
\def\fintegral{\fint}

\def\mm{\mathcal{M}}
\def\mr{\mathcal{R}}
\def\aa{\mathcal{A}}
\def\ll{\mathcal{L}}

\nc\pa{\partial}
\def\pmd{p-\delta}

\nc\CC{\mathbb{C}}
\nc\RR{\mathbb{R}}
\nc\QQ{\mathbb{Q}}
\nc\ZZ{\mathbb{Z}}
\nc\NN{\mathbb{N}}

\def\bea{\begin{equation}\begin{aligned}}
\def\ena{\end{aligned}\end{equation}}

\def\beas{\begin{equation*}\begin{aligned}}
\def\enas{\end{aligned}\end{equation*}}

\nc\m[1]{\left| #1\right|}
\nc\norm[1]{\left\|#1\right\|}
\nc\axgrad[1]{\mathcal{A}(x, #1)}

\newtheorem{theorem}{Theorem}[section]
\newtheorem{lemma}[theorem]{Lemma}
\newtheorem{corollary}[theorem]{Corollary}
\newtheorem{proposition}[theorem]{Proposition}

\newtheorem{definition}[theorem]{Definition}
\newtheorem{remark}[theorem]{Remark}        
\numberwithin{equation}{section}

\usepackage[ddmmyyyy]{datetime}

\def\gbmo{$(\ga, \, R_0)$-BMO }
\def\gflt{$(\ga, \, R_0)$-Reifenberg flat }

\def\omegarho{\Omega_{\rho}}

\def\fpmd{\frac{1}{\pmd}}

 \def\aw{  A}
\newcommand{\awpq}[2]{\aw_{\frac{#1}{#2}}}

\def\dv{\mathop{\rm div}}
\DeclareMathOperator{\diam}{diam}
\begin{document}
\title[An end-point weighted estimate for quasilinear equations]{An end-point global gradient weighted  estimate  for quasilinear equations in non-smooth domains}

\author[Karthik Adimurthi]
{Karthik Adimurthi}
\address{Department of Mathematics,
Louisiana State University,
303 Lockett Hall, Baton Rouge, LA 70803, USA.}
\email{kadimu1@math.lsu.edu \and karthikaditi@gmail.com}

\author[Nguyen Cong Phuc]
{Nguyen Cong Phuc}
\address{Department of Mathematics,
Louisiana State University,
303 Lockett Hall, Baton Rouge, LA 70803, USA.}
\email{pcnguyen@math.lsu.edu}

\begin{abstract}
A weighted norm inequality involving $A_1$ weights is obtained at the natural exponent for gradients of solutions to quasilinear elliptic equations in 
Reifenberg flat domains. Certain gradient estimates in Lorentz-Morrey spaces below the natural exponent are also obtained as a consequence of our analysis. 
\end{abstract}


\maketitle


\section{Introduction}
One of the main goals of this paper is to obtain global gradient weighted estimates of the form 
\begin{equation}\label{A1-bound}
\int_\Om |\nabla u|^p w dx \leq C \int_\Om |{\bf f}|^p w dx
\end{equation}
for weights $w$ in the Muckenhoupt class $A_1$ and for solutions $u$ to the nonhomgeneous nonlinear boundary value problem
\begin{eqnarray}\label{basicpde}
\left\{ \begin{array}{rcl}
 \text{div}\,\aa(x, \nabla u)&=&\text{div}~|{\bf f}|^{p-2}{\bf f}  \quad \text{in} ~\Omega, \\
u&=&0  \quad \text{on}~ \partial \Omega.
\end{array} \right.
\end{eqnarray}
Here $p>1$ and $\text{div}\,\aa(x, \nabla u)$ is modelled after the standard $p$-Laplcian $\Delta_p u:= \text{div} (|\nabla u|^{p-2} \nabla u)$. Aslo, ${\bf f}$ is a given vector  field defined in a bounded domain $\Om$ that may have a non-smooth boundary.

More specifically, in \eqref{basicpde} the nonlinearity $\aa : \RR^n \times \RR^n \rightarrow \RR^n$ is a 
Carath\'edory vector valued function, i.e., $\aa(x,\xi)$ is measurable in $x$ for every $\xi$ and continuous in 
$\xi$ for a.e. $x\in\RR^n$. 
We also assume that $\aa(x,0) = 0$ and $\mathcal{A}(x, \xi)$ is continuously differentiable in $\xi$ away from the origin for a.e. $x \in \RR^n$. For our purpose, 
we  require that $\aa$ satisfy the following monotonicity and Lipschitz type conditions: for some $p>1$, there holds
\bea
\label{monotone}
  \langle\aa(x,\xi)-\aa(x,\eta),\xi-\eta \rangle\geq
\Lambda_0 (|\xi|^2+|\eta|^2)^{\frac{p-2}{2}}|\xi-\eta|^2 
\ena
and 
\bea
\label{ellipticity}
|\aa(x,\xi) - \aa(x,\eta)| \leq \Lambda_1 |\xi - \eta| (|\xi|^2+|\eta|^2)^{\frac{p-2}{2}}
\ena
for every  $(\xi, \eta)\in \RR^n \times \RR^n\setminus (0,0)$ and a.e. $x \in\RR^n$. Here $\La_0$ and $\La_1$ are positive constants.
Note that \eqref{ellipticity} and the assumption $\aa(x,0)=0$ for a.e. $x \in \RR^n$ imply the following growth condition
\beas
\label{bound-up}
|\aa(x,\xi) | \leq \La_1 |\xi|^{p-1}. 
\enas

Our additional regularity assumption on the nonlinearity $\aa$ is the following \gbmo condition. 
To formulate it, for each ball $B$, we let 
\beas
\overline{\aa}_B(\xi) = \fintegral_B \aa(x,\xi) \, dx = \frac{1}{|B|} \integral_B \aa(x,\xi)\, dx , 
\enas
and define the following function that measures the oscillation of $\aa(\cdot, \xi)$ over $B$ by
\beas
\Upsilon(\aa,B) (x) := \sup_{\xi \in \mathbb{R}^{n}\setminus \{0\}} \frac{|\aa({x, \xi}) - \overline{\aa}_{B}({\xi})|}{|\xi|^{p-1}}.
\enas

\begin{definition}
\label{BMO-condition}
 Given two positive numbers $\ga$ and $R_0$, we say that $\aa(x,\xi)$ satisfies a \gbmo condition  with exponent $\tau>0$ if
\beas
\left[ \aa \right]_\tau^{R_0} := \sup\limits_{y \in \RR^n,\, 0<r\leq R_0} \left( \fintegral_{B_r(y)} \Upsilon(\aa,\, B_r(y)))(x)^\tau \, dx \right)^{\frac{1}{\tau}} \leq \ga. 
\enas
\end{definition}

In the linear case, where $\aa(x,\xi) = A(x)\xi$ for an elliptic matrix $A$, we see that
\beas
\Upsilon(\aa,B)(x) \leq |A(x) - \overline{A}_B|
\enas
for a.e. $x \in \RR^n$, and thus Definition \ref{BMO-condition} can be viewed as a natural extension of the standard small BMO condition to the nonlinear setting. For general nonlinearities $\aa(x,\xi)$ of at most linear growth, i.e., $p=2$, the above \gbmo condition was introduced in \cite{BW1}, whereas such a condition for general $p>1$ appears first in \cite{Ph3}.  We remark that the \gbmo condition allows the nonlinearity $\aa(x,\xi)$ to have certain discontinuity in $x$, and it can be used as an appropriate substitute for the Sarason VMO condition (vanishing mean oscillation \cite{Sa}, see also \cite{BW2,BW1,GP,IKM,Mil,Se,VMR}).

The domain over which we solve our equations may be non-smooth but should satisfy some flatness condition. Essentially, at each boundary point and every scale, we require the boundary of the domain to be between two hyperplanes separated by a distance proportional to the scale. Absence of such flatness may result in a limited regularity of the solutions, as demonstrated in the counterexample given in  \cite{M-P1} (see also \cite{JK}).

\begin{definition}
 \label{reifenberg_flat}
Given $\ga \in (0,1)$ and $R_0 > 0$, we say that $\Omega$ is a \gflt domain if for every $x_0 \in \partial \Omega$ and every $r \in (0,R_0]$, there exists a system of coordinates $\{y_1,y_2,\ldots,y_n\}$, which may depend on $r$ and $x_0$, so that in this coordinate system $x_0 = 0$ and that 
\beas
B_r(0) \cap \{y_n > \ga r\} \subset B_r (0) \cap \Omega \subset B_r(0) \cap \{y_n > -\ga r\}.
\enas
\end{definition}

For more on Reifenberg flat domains and their many applications, we refer to the papers \cite{HM,Jon,KT1,KT2,Rei,Tor}. We mention here that Reifenberg flat domains can be very rough. They include Lipschitz domains with sufficiently small Lipschitz constants (see \cite{Tor}) and even some domains with fractal boundaries. In particular, all $C^1$ domains are included in this paper.
\begin{remark}\label{domainremark}
{\rm
If $ \Omega$ is a \gflt domain with $\gamma< 1$, then  for any
point $x \in \partial \Om$ and $ 0<\rho<R_0(1-\gamma)$ there exists a coordinate system  $\{z_{1}, z_{2}, \cdots, z_{n}\}$
with the origin 0 at some point in the interior of $\Omega$ such that in this coordinate system $x=(0, \dots, 0, -\gamma' \rho)$ and
\[
B_{\rho}^{+}(0) \subset \Omega \cap B_{\rho}(0) \subset B_{\rho}(0)\cap \{(z_1, z_2, \dots, z_n):z_{n} > -2\rho\gamma'\}, 
\]
   where  $\gamma' = \gamma/(1-\gamma)$ and $B_{\rho}^{+}(0):=B_{\rho}(0)\cap \{(z_1, \dots, z_n): z_n>0\}$. Thus, if $\gamma<1/2$ then
\[
B_{\rho}^{+}(0) \subset \Omega \cap B_{\rho}(0) \subset B_{\rho}(0)\cap \{(z_1, z_2, \dots, z_n):z_{n} > -4\rho\gamma\}. 
\]   
}
\end{remark}

Now we shall collect some properties of weights. In this paper, we shall only be concerned with Muckenhoupt weights.  By an  $A_s$ weight, $1<s<\infty$, we mean a nonnegative function
$w \in L^{1}_{\rm loc}(\mathbb{R}^n{})$ such that  the quantity
\[
[w]_{s} := \sup_{B} \left (\fintegral_{B} w(x)\, dx\right )\left (\fintegral_{B} w(x)^{\frac{-1}{s-1}}\, dx\right )^{s-1}  < +\infty,
\]
where the supremum is taken over all balls $B \subset \mathbb{R}^{n}$. For $s=1$, we say that $w$ is an $A_1$ weight if 
\[
[w]_{1} := \sup_{B} \left (\fintegral_{B} w(x)\, dx\right ) \norm{w^{-1}}_{L^\infty(B)}  < +\infty.
\]

The quantity $[w]_{s}$, $1\leq s <\infty$, will be referred to as the $A_{s}$ constant of $w$.
The $A_s$ classes are increasing, i.e., $A_{s_1}\subset A_{s_2}$ whenever $1\leq s_1<s_2<\infty$.   
A broader class of weights is the $\aw_{\infty}$ weights which,  by definition, is the union of all $\aw_{s}$ weights for $1\leq s< \infty$.
The following  characterization of $\aw_{\infty}$ weights will be needed later (see  \cite[Theorem 9.3.3]{Gra}).
\begin{lemma}\label{inversedoubling}
 A weight  $w \in \aw_{\infty}$ if and only if there are constants $\Xi_0, \Xi_1 > 0$ such that for every ball $B \subset \mathbb{R}^{n}$ and every measurable subsets $E$ of $B$
\begin{eqnarray}\label{inversedoublinginequality}
w(E)\leq \Xi_0 \left ( \frac{|E|}{|B|}\right)^{\Xi_1}w(B).
\end{eqnarray}
Moreover, if $w$ is an $\aw_s$ weight  with $[w]_s\leq \overline{\om}$ then the constants $\Xi_0$ and $\Xi_1$ above can be chosen so that 
$\max\{\Xi_0, 1/\Xi_1\} \leq c(\overline{\om}, n)$.
\end{lemma}

In \eqref{inversedoubling}, the notation  $w(E)$  stands for the integral $\integral_{E}w(x)\, dx$, and likewise for $w(B)$, etc.
Henceforth, we will use this notaion without further explanation. Also, we will refer to 
$(\Xi_0,\Xi_1)$ as a pair of $\aw_{\infty}$ constants of $w$ provided they satisfy \eqref{inversedoublinginequality}.

We  now  recall the definition of weighted Lorentz spaces. For a general weight $w$, the weighted Lorentz space $L_w(s,t)(\Om)$ with $0<s<\infty$, $0<t\leq \infty$, is the set of measurable functions $g$ on $\Om$ such that 
\begin{equation*}
\|g\|_{L_w(s,\, t)(\Omega)} :=
\left[s \integral_{0}^{\infty}\alpha^t (w(\{x\in\Om: |g(x)|>\alpha\}))^{\frac{t}{s}} \frac{d\alpha}
{\alpha}\right]^{\frac{1}{t}} <+ \infty
\end{equation*}
when $t\neq\infty$; for $t=\infty$ the space $L_w(s,\, \infty)(\Om)$ is set to be the usual Marcinkiewicz space with quasinorm
\beas
\|g\|_{L_w(s,\, \infty)(\Omega)}:=\sup\limits_{\alpha >0} \, \alpha w(\{x\in \Om: |g(x)|>\alpha\})^{\frac{1}{s}}.
\enas
It is easy to see that when $t=s$ the weighted Lorentz space $L_w(s,\, s)(\Om)$ is nothing but the weighted Lebesgue space $L^{s}_{w}(\Om)$, which is equivalently defined as 
\beas
g\in L^{s}_{w}(\Omega) \Longleftrightarrow \integral_{\Omega}|g(x)|^s w(x)dx < +\infty.
\enas 
As usual, when $w\equiv 1$ we  simply write $L(s,\, t)(\Om)$ instead of $L_w(s,\, t)(\Om)$.

 A function $g\in L(s,t)(\Om)$, $0<s<\infty$, $0<t\leq \infty$
is said to belong to the Lorentz-Morrey function space $\mathcal{L}^{\theta}(s,t)(\Om)$ for some $0<\theta\leq n$, if
\begin{equation*}
\norm{g}_{\mathcal{L}^{\theta}(s,t)(\Om)} := \sup_{ \substack{0<r \leq {\rm diam} (\Om), \\ z\in\Om   }}
r^{\frac{\theta-n}{s}}\norm{g}_{L(s,t)(B_{r}(z)\cap\Om)}  <+\infty.
\end{equation*}

When $\theta=n$, we have $\mathcal{L}^{\theta}(s,t)(\Om)=  L(s,t)(\Om)$. Moreover, when $s=t$ the space  $\mathcal{L}^{\theta}(s,t)(\Om)$ becomes the usual Morrey space based on $L^s$ space. 

A basic use of Lorentz spaces is to improve the classical Sobolev Embedding Theorem. For example, if $f\in W^{1,q}$ for some 
$q\in (1,n)$  then $$f\in L(nq/(n-q), q)$$ (see, e.g., \cite{Zie}), which is better than the classical result  
$$f\in L^{nq/(n-q)}=L(nq/(n-q), nq/(n-q))$$
since $L(s,t_1)\subset L(s, t_2)$ whenever $t_1\leq t_2$.  Another use of Lorentz spaces is to capture logarithmic singularities. For example,  for any $\beta>0$ we have 
$$\frac{1}{|x|^{n/s} (\log|x|)^\beta}\in L(s,t)(B_1(0)) \quad {\rm if~and~only~if~} t>\frac{1}{\beta}.$$ 
Lorentz spaces have also been used successfully  in improving regularity criteria for the full
3D Navier-Stokes system of equations (see, e.g., \cite{Soh}). 

On the other hand, Lorentz-Morrey spaces are neither rearrangement invariant spaces, nor interpolation spaces. 
They often show up in the analysis of  Schr\"odinger operators via the so-called Fefferman-Phong condition (see \cite{Fef}), or in the regularity theory of nonlinear equations of fluid dynamics.

We are now able to state the main result of the paper.
\begin{theorem}\label{main_theorem}
Suppose that $\aa$ satisfies \eqref{monotone}-\eqref{ellipticity}. Let $t\in (0, \infty]$, $q\geq p$, and let $w$ be an $A_{q/p}$ weight. There exist  constants $\tau=\tau(n,p,\La_0, \La_1)>1$ and $\ga=\ga(n,p,\La_0, \La_1, q, [w]_\infty)>0$  such that the following holds.
If $u \in W_0^{1,p}(\Om)$ is a  solution of \eqref{basicpde} in a \gflt domain $\Om$ with $[\mathcal{A}]_\tau^{R_0}\leq \ga$,
then one has the  estimate 
\beas
\|\nabla u\|_{L_w(q,t)(\Om)} \leq C \|\bff \|_{L_w(q,t)(\Om)},
\enas
where  the  constant $C = C(n,p,\La_0,\La_1, q, t, [w]_{q/p}, {\rm diam}(\Om)/R_0)$.
\end{theorem}

\begin{remark}
\label{remark_main}
{\rm By Remark \ref{upper-const} below and Lemma \ref{inversedoubling}, it follows that if $\overline{\om}$ is an upper bound for $[w]_{q/p}$, i.e.,  $[w]_{q/p} \leq \overline{\omega}$, then the constants $C$ and $\gamma$ above can be chosen to depend on $\overline{\omega}$ instead of $[w]_{q/p}$ or $[w]_{\infty}$.
}
\end{remark}

Theorem \ref{main_theorem}  follows from  Theorem \ref{main_theorem-M} below (applied with $M=q$)  and the boundedness 
property of the Hardy-Littlewood maximal function on weighted spaces. 
Its main contribution is the end-point case $q=p$, which yields inequality 
\eqref{A1-bound} for all  $A_1$ weights $w$ as proposed earlier.
The case $q>p$ has been obtained in \cite{M-P2, MNP} but the proofs in those papers can only yield a weak-type bound at the end-point $q=p$.

Theorem \ref{main_theorem-M} also yields the following gradient estimate below the natural exponent for {\it very weak} solutions, i.e., distributional solutions 
that may not have  finite $L^p$ energy.
\begin{theorem}\label{Lorent_Morrey}
Suppose that $\aa$ satisfies \eqref{monotone}-\eqref{ellipticity} and let $\theta_0$ be a fixed number in $(0,n]$.  There exist   $\tau=\tau(n,p,\La_0, \La_1)>1$, $\delta=\delta(n,p, \theta_0, \La_0, \La_1)>1$, and $\ga=\ga(n,p,\theta_0,\La_0, \La_1)>0$  such that the following holds.
If $u \in W_0^{1,p-\delta}(\Om)$ is a very weak solution of \eqref{basicpde} in a \gflt domain $\Om$ with $[\mathcal{A}]_\tau^{R_0}\leq \ga$ and $\bff \in \ll^{\theta}(q,t)(\Om,\RR^n)$,
then there holds:
\bea \label{L-bound-below}
\|\nabla u\|_{\ll^{\theta}(q,t) (\Om)} \leq C \|\bff\|_{\ll^{\theta}(q,t) (\Om)}
\ena
for all $q \in (\pmd, p]$, $0 < t \leq \infty$ and $\theta_0\leq  \theta \leq n$. Here the constant $C = C(n,p,q,t,\theta_0, \La_0,\La_1,\diam(\Om)/R_0)$.
\end{theorem}

The proof of Theorem \ref{Lorent_Morrey} follows by first applying Theorem \ref{main_theorem-M} with $M=p$ and the weight functions 
\beas
w(x) = \min \{ |x-z|^{-n+\theta - \rho} , r^{-n+\theta-\rho} \},
\enas
for any $z\in\Om$ and $r\in (0, {\rm diam}(\Om)]$ and a fixed  $\rho \in (0,\theta)$. 
Note that $w$ is an $\aw_1$ weight with its $\aw_1$ constant $[w]_1$ being bounded from above by a constant independent of $z$ and $r$. See also Remark \ref{upper-const}.
The rest of the proof then follows verbatim as in that of \cite[Theorem 2.3]{M-P2}. We mention that the sub-natural bound \eqref{L-bound-below}
was also obtained in our earlier work \cite{AN} but with the restriction  $\theta \in [p-2\delta,n]$, and in \cite{IW} with $\theta=n$, i.e., for pure Lebesgue spaces only. Note also that the super-natural case $q>p$ has been 
obtained in \cite{M-P2, MNP}.

Unweighted estimate of the form 
\begin{equation}\label{unweighted-IW}
\int_\Om |\nabla u|^q dx \leq C \int_\Om |{\bf f}|^q dx
\end{equation}
for solutions $u$ to \eqref{basicpde} in the full sub-natural range $q\in(p-1, p)$ is currently a wide open problem (even for smooth domains and the standard $p$-Laplacian). This is essentially known as a conjecture of 
T. Iwaniec who originally stated it for $\Om=\RR^n$ and $q\in (\max\{1, p-1\}, p)$ in \cite{I}.  For the super-natural case $q\geq p$, we refer to the papers \cite{I, KZ1, KZ2}  and  \cite{BW3, CP}.
For $q\in[p-\delta,p)$ with a small $\delta>0$, see \cite{AN, IW}.

This conjecture is another motivation for us to consider 
weighted estimates of the form \eqref{A1-bound} at the natural exponent $p$. In fact, using the extrapolation theory of  Garc\'ia-Cuerva and 
Rubio de Francia (see \cite{GR} and \cite[Chapter 9]{Gra}) we see that if the weighted bound  \eqref{A1-bound} holds for all weights 
$w\in A_{\frac{p}{p-1}}$ then the unweighted bound \eqref{unweighted-IW} will follow for all $q\in(p-1, p)$. More precisely, we have the following more general result, whose complete proof will be given in the Appendix.
\begin{theorem}
\label{theorem_extrapolation}
 For $p>1$, let $\bff \in L^p(\Om,\RR^n)$ be a given vector field  and denote $u \in W_0^{1,p}(\Om)$ to be the unique weak solution to \eqref{basicpde}.
 Suppose we have that 
\bea
\label{weight11}
 \integral_{\Om} |\nabla u|^p \, v(x) \, dx \leq C([v]_{\frac{p}{p-1}}) \integral_{\Om} |\bff|^p \, v(x) \, dx 
\ena
holds for all weights $v \in \awpq{p}{p-1}$.  Then for any $p-1<q<\infty$, there holds
\bea
\label{extrapolation_result}
\integral_{\Om} |\nabla u|^{q} w(x) \, dx \leq C([w]_{\frac{q}{p-1}}) \integral_{\Om} |\bff(x)|^{q} w(x) \, dx\,  
\ena
for all weights  $w \in \awpq{q}{p-1}$.
\end{theorem}
What we obtain in this paper is the weighted bound \eqref{A1-bound} for all weights $w\in A_1$ which unfortunately is not enough for us to apply the above extrapolation theorem. However, it provides us with an alternative view on the conjecture of T. Iwaniec and gives us a different sense of how far we are from completely resolving this conjecture. Of course, one can also generalize this conjecture  by proposing  the bound \eqref{weight11} for  
all weights $v \in \awpq{p}{p-1}$.

\noindent {\bf Notation:}  Throughout the paper, we shall write $A\apprle B$ to denote $A\leq c\, B$ for a positive constant 
$c$ independent of the parameters involved. Basically, $c$ is allowed to depend only on $n, p,  \Lambda_0, \Lambda_1, \ga$ and $R_0$.

\section{Local difference estimates}\label{sec4}

In this section, we obtain certain local interior and boundary difference estimates that are essential to our global estimates later. 

\subsection{Interior estimates}
 Let $u \in W^{1,\pmd}_{0}(\Om)$ for some $\de \in (0,\min\{1,p-1\})$   be a very weak solution to the equation
\begin{equation}\label{basicpde1}
\text{div}\,\aa(x, \nabla u)=\text{div}~|{\bf f}|^{p-2}{\bf f}
\end{equation} 
in a domain $\Om$. For each ball $B_{2R}=B_{2R}(x_0) \Subset \Omega$, we let $w \in u + W_0^{1,\pmd}(B_{2R})$ be a very weak solution to the  problem 
\begin{equation}
\label{firstapprox}
\left\{ \begin{array}{rcl}
 \dv \aa(x, \nabla w)&=&0   \quad \text{in} ~B_{2R} \\ 
w&=&u  \quad \text{on}~ \partial B_{2R}.
\end{array}\right. 
\end{equation} 

For sufficiently small $\delta$, the existence of such $w$ follows from the result of  \cite[Theorem 2]{IW}. The following theorem tells more on the integrability property of $w$ and its relation to $u$ by means of a comparison estimate.

\begin{theorem}\label{interior_higher_integrability} 
Under  \eqref{monotone} and \eqref{ellipticity}, there exists a small number  $\delta_0=\delta_0(n,p,\Lambda_0, \Lambda_1)>0$ such that the following holds for any $\delta\in(0,\delta_0)$. 
Given any $u \in W_{\rm loc}^{1,\pmd}(\Om)$ solving \eqref{basicpde1} and any $w$  as in \eqref{firstapprox}, we have the following comparison estimate
\begin{equation*}
\fintegral_{B_{2R}} |\nabla u-\nabla w|^{p-\delta}\, dx  \leq C\,   \delta^{\frac{p-\delta}{p-1}}  \fintegral_{B_{2R}} |{\nabla u }|^{p-\delta}\, dx  +    \fintegral_{B_{2R}} |{\bf f}|^{p-\delta}\, dx 
\end{equation*}
if $p\geq 2$, and
\beas
\fintegral_{B_{2R}} |\nabla u-\nabla w|^{p-\delta}\, dx  & \leq C \, \delta^{p-\delta}  \fintegral_{B_{2R}} |\nabla u|^{p-\delta}\, dx  \, + \\
&  + C\bigg{(}\fintegral_{B_{2R}} |\bff|^{p-\delta} \, dx\bigg{)}^{p-1} 
\bigg{(} \fintegral_{B_{2R}} |\nabla u|^{p-\delta} \, dx \bigg{)}^{2-p} 
\enas
if $1<p<2$. Moreover, 
\begin{equation}\label{IS-pmd}
\integral_{B_{2R}} |\nabla w|^{\pmd}\,  dx  \leq C   \integral_{B_{2R}} |\nabla u|^{\pmd} \, dx,
\end{equation}
and for any ball $B_r(y)\subset B_{2R}$
\begin{equation}\label{HI-IS-L}
 \Big( \integral_{B_{r/2}(y)} |\nabla w|^{p+\de_0} \, dx \Big)^{\frac{1}{p+\de_0}} \leq C \Big( \integral_{B_{r}(y)} |\nabla w|^{p-\delta_0} \, dx. \Big)^{\frac{1}{p-\delta_0}}
\end{equation}
Here the constants $C$ depend only on $n, p, \Lambda_0$ and $\Lambda_1$.
\end{theorem}

The bound \eqref{IS-pmd} was obtained in \cite[Theorem 2]{IW}. The higher integrability result, inequality \eqref{HI-IS-L}, was proved in 
\cite[Theorem 1]{IW} (see also \cite{John}). On the other hand, the comparison estiamte above has been obtained  just recently in \cite[Lemma 2.8]{AN}.

Now with $u$ as in \eqref{basicpde} and $w$ as in \eqref{firstapprox},  we further define another function $v\in w+ W_{0}^{1,\,p}(B_{R})$
as the unique solution to the Dirichlet problem
\begin{equation}\label{secondapprox}
\left\{ \begin{array}{rcl}
 \dv  \overline{\aa}_{B_{R}}(\nabla v)&=&0   \quad \text{in} ~B_{R}, \\ 
v&=&w  \quad \text{on}~ \partial B_{R},
\end{array}\right.
\end{equation}
where $B_R=B_R(x_0)$. This equation makes sense since we have good regularity for $w$ as a consequence of Theorem \ref{interior_higher_integrability}. We shall now prove another useful interior difference estimate.

\begin{lemma}\label{BMOapprox1}
Under \eqref{monotone}-\eqref{ellipticity},  let $\de \in (0,\de_0)$, where $\de_0$ is as in Theorem \ref{interior_higher_integrability}. Let  $w$ and $v$ be as in \eqref{firstapprox} and \eqref{secondapprox}. 
For $\tau =\frac{p}{\de_0}\frac{(p+\de_0)}{ (p-1)}$, there exists a constant $C=C(n,p,\La_0,\La_1)$ such that 
\begin{align*}
\fintegral_{B_{R}} |\nabla v-\nabla w|^{\pmd}\, dx \leq C \Big(\fintegral_{B_{R}} \Upsilon(\aa,B_{R})(x)^{\tau}\, dx\Big)^{\min\{\pmd, \frac{\pmd}{p-1}\}/\tau} \times \\ \times \Big(\fintegral_{B_{2R}}|\nabla w|^{\pmd} \, dx\Big) .
\end{align*}
\end{lemma}
\begin{proof}
Using \eqref{monotone} and the fact that both $v$ and $w$ are solutions,  we have
\begin{align*}
\fintegral_{B_R}(|\nabla v|^2+ |\nabla w|^2&)^{\frac{p-2}{2}}| \nabla w- \nabla v|^2 \, dx  \\
 &\apprle   \fintegral_{B_R} \langle \overline{\aa}_{B_R}(\nabla w)-  \overline{\aa}_{B_R}(\nabla v), \nabla w-\nabla v\rangle \, dx\\
 &= C   \fintegral_{B_R} \langle \overline{\aa}_{B_R}(\nabla w)-  \aa(x,\nabla w), \nabla w-\nabla v\rangle \, dx\\
 &\apprle   \fintegral_{B_R} \Upsilon(\aa,B_R)(x) |\nabla w|^{p-1} |\nabla w-\nabla v|  \, dx.
\end{align*}

Using H\"older's inequality with exponents $p$, $\frac{p+\de_0}{p-1}$, and $\tau$ we get
\begin{equation}
 \label{random2}
\begin{aligned}
\fintegral_{B_R}(|\nabla v|^2+ &|\nabla w|^2)^{\frac{p-2}{2}}|  \nabla w- \nabla v|^2 \, dx \\
& \apprle    \left(\fintegral_{B_{R}}\Upsilon(\aa,B_{R})(x)^{\tau}\, dx\right)^{\frac{1}{\tau}} \left(\fintegral_{B_R}|\nabla w|^{p+\de_0}\,  dx\right)^{\frac{p-1}{p+\de_0}} \times \\
&  \qquad \qquad \times \left(\fintegral_{B_R} |\nabla w-\nabla v|^{p} \, dx\right)^{\frac{1}{p}}\\
& \apprle    \left(\fintegral_{B_{R}}\Upsilon(\aa,B_{R})(x)^{\tau}\, dx\right)^{\frac{1}{\tau}} \left(\fintegral_{B_{2R}}|\nabla w|^{p-\delta}\,  dx\right)^{\frac{p-1}{p-\delta}} \times \\
&  \qquad \qquad \times \left(\fintegral_{B_R} |\nabla w-\nabla v|^{p} \, dx\right)^{\frac{1}{p}},\\
%
\end{aligned}
\end{equation}
where the last inequality follows from \eqref{HI-IS-L} of Theorem \ref{interior_higher_integrability}. 

Thus for $p\geq 2$, using pointwise estimate $$|\nabla w- \nabla v|^p \leq (|\nabla v|^2+ |\nabla w|^2)^{\frac{p-2}{2}}| \nabla w- \nabla v|^2,$$ we find
$$\left(\fintegral_{B_R} |\nabla w-\nabla v|^{p} \, dx\right)^{\frac{p-1}{p}} \apprle
    \left(\fintegral_{B_{R}}\Upsilon(\aa,B_{R})^{\tau}\, dx\right)^{\frac{1}{\tau}} \left(\fintegral_{B_{2R}}|\nabla w|^{p-\delta}\,  dx\right)^{\frac{p-1}{p-\delta}}.$$

By H\"older's inequality this yields the desired estimate in the case $p\geq 2$.

For $1<p<2$ we write
\beas
|\nabla v-\nabla w|^{p}= (|\nabla v|^2+ |\nabla w|^2)^{\frac{(p-2)p}{4}} |\nabla w- \nabla v|^p (|\nabla v|^2+ |\nabla w|^2)^{\frac{(2-p)p}{4}},
\enas 
and apply  H\"older's inequality with exponents $\frac{2}{p}$ and $\frac{2}{2-p}$ to obtain
\begin{align*}
\fintegral_{B_R}| \nabla w- \nabla v|^p \, dx &\leq \left(\fintegral_{B_R}(|\nabla v|^2+ |\nabla w|^2)^{\frac{p-2}{2}}|  \nabla w- \nabla v|^2 \, dx\right)^{\frac{p}{2}}\times \nonumber\\
&\qquad  \times\left(\fintegral_{B_R}(|\nabla v|^2+ |\nabla w|^2)^{\frac{p}{2}} \,dx\right)^{\frac{2-p}{2}} \nonumber \\
\label{equn1}&\apprle    \left(\fintegral_{B_{R}}\Upsilon(\aa,B_{R})^{\tau}\, dx\right)^{\frac{p}{2\tau}} \left(\fintegral_{B_{2R}}|\nabla w|^{p-\delta}\, dx\right)^{\frac{(p-1)p}{(p-\delta)2}}\times \\
&\qquad \times\left(\fintegral_{B_R} |\nabla w-\nabla v|^{p}\, dx\right)^{\frac{1}{2}}\left(\fintegral_{B_{R}}|\nabla w|^{p}\, dx\right)^{\frac{2-p}{2}}. \nonumber
\end{align*}
Here we used \eqref{random2} and the easy energy bound $\int_{B_R} |\nabla v|^pdx \leq c \int_{B_R}|\nabla w|^p dx$ in 
the last inequality.
Using \eqref{HI-IS-L} of Theorem \ref{interior_higher_integrability}, this yields
\begin{align*}
\fintegral_{B_R}| \nabla w- \nabla v|^p \, dx  &\apprle    \left(\fintegral_{B_{R}}\Upsilon(\aa,B_{R})^{\tau}\, dx\right)^{\frac{p}{\tau}} \left(\fintegral_{B_{2R}}|\nabla w|^{p-\delta}\, dx\right)^{\frac{p}{p-\delta}}.
\end{align*}

Now an application of H\"older's inequality  gives the desired estimate in the case $1<p<2$.
%
\end{proof}

\begin{corollary} \label{mainlocalestimate-interior} 

Under \eqref{monotone}-\eqref{ellipticity},  let $\tau =\frac{p}{\de_0}\frac{(p +\de_0)}{ (p-1)}$ and  $\de \in (0,\de_0)$, where $\de_0$ is   as in Theorem 
 \ref{interior_higher_integrability}. . 
Then for any $\epsilon>0$, there exists  $\ga=\ga(\epsilon)>0$ such that if $u\in W^{1,\, \pmd}_0(\Om)$ is a
very weak solution of  \eqref{basicpde} satisfying
\begin{gather*}
\fintegral_{B_{2R}}|\nabla u|^{\pmd}\, dx \leq 1,\, \fintegral_{B_{2R}} |{\bf f}|^{\pmd}\, dx \leq \ga^{\pmd}, {\rm ~and~}
\fint_{B_R}\Upsilon(\mathcal{A}, B_R)^\tau dx \leq \ga^{\tau}, 
\end{gather*}
%
for a ball $B_{2R} \Subset \Om$, then there exists $v\in W^{1,\, p}(B_R)\cap W^{1,\, \infty}(B_{R/2})$ such that 
$$\fintegral_{B_{R}} |\nabla u-\nabla v|^{\pmd}dx \leq \epsilon^{\pmd}, {\rm ~and~} \norm{\nabla v}_{L^{\infty}(B_{R/2})}\leq C_0 = C_0(n,p,\La_0,\La_1).$$
\end{corollary}
\begin{proof} Let $w$ and $v$ be as in \eqref{firstapprox} and \eqref{secondapprox} respectively. Since we have $v \in W^{1,p}(B_R)$, standard regularity theory gives (see, e.g., \cite{Tol})
\begin{eqnarray*}
\norm{\nabla v}_{L^{\infty}(B_{R/2})}^p&\apprle& \fintegral_{B_{R}}|\nabla v|^p dx \apprle \fintegral_{B_{R}}|\nabla w|^p dx \\
&\apprle& \left( \fintegral_{B_{2R}}|\nabla w|^{\pmd} dx\right)^{\frac{p}{\pmd}} \apprle \left( \fintegral_{B_{2R}}|\nabla u|^{\pmd} dx\right)^{\frac{p}{\pmd}} \leq C_0.
\end{eqnarray*}
Here we applied  Theorem \ref{interior_higher_integrability}. 
The proof of the corollary now follows from the comparison estimate in Theorem \ref{interior_higher_integrability} and  Lemma \ref{BMOapprox1}.
\end{proof}

\subsection{ Boundary estimates}  
We now consider the corresponding local  estimates near the boundary.  Suppose that the domain $\Om$ is \gflt with $\ga < 1/2$. Let $x_0 \in \partial \Om$, $R \in (0,R_0/20)$, and let $u \in W_0^{1,\pmd} (\Om)$ be a very weak solution to \eqref{basicpde} for some $\de \in (0,\min\{1,p-1\})$. On $\Om_{20R}=\Om_{20R} (x_0) = B_{20R}(x_0) \cap \Om$, we let $w \in u+ W_0^{1,\pmd} (\Om_{20R}(x_0))$ be a very weak solution to the  problem: 
\begin{equation}\label{wapprox}
\left\{ \begin{array}{rcl}
 \dv \aa(x, \nabla w)&=&0   \quad \text{in} ~\Om_{20R}, \\ 
w&=&u  \quad \text{on}~ \partial \Om_{20R}.
\end{array}\right.
\end{equation}
We now extend $u$ by zero to $\RR^n\setminus \Om$ and then extend $w$ by $u$ to $\RR^n\setminus \Om_{20R}(x_0)$.  Analogous to Theorem \ref{interior_higher_integrability}, we have the following boundary  counterpart. 

\begin{theorem} \label{higher-integrability-boundary}
Under  \eqref{monotone} and \eqref{ellipticity}, there exists a small number  $\tilde{\delta}_0=\tilde{\delta}_0(n,p,\Lambda_0, \Lambda_1, \gamma)>0$ such that the following holds for any $\delta\in(0,\tilde{\delta}_0)$. 
For any $u \in W_0^{1,\pmd}(\Om)$ solving \eqref{basicpde} and any $w$  as in \eqref{wapprox}, after extending $\bff$ and $u$ outside $\Om$ by zero and $w$ by $u$ outside $\Om_{20R}$, we have the following comparison estimate
\begin{equation*}
\fintegral_{B_{20R}} |\nabla u-\nabla w|^{p-\delta}\, dx  \leq C\,   \delta^{\frac{p-\delta}{p-1}}  \fintegral_{B_{20R}} |{\nabla u }|^{p-\delta}\, dx  +    \fintegral_{B_{20R}} |{\bf f}|^{p-\delta}\, dx 
\end{equation*}
if $p\geq 2$, and
\beas
\fintegral_{B_{20R}} |\nabla u-\nabla w|^{p-\delta}\, dx  & \leq C  \, \delta^{p-\delta}  \fintegral_{B_{20R}} |\nabla u|^{p-\delta}\, dx  \, + \\
& + C\bigg{(}\fintegral_{B_{20R}} |\bff|^{p-\delta} \, dx\bigg{)}^{p-1} 
\bigg{(} \fintegral_{B_{20R}} |\nabla u|^{p-\delta} \, dx \bigg{)}^{2-p} 
\enas
if $1<p<2$. Moreover, 
\begin{equation}\label{pmd-boundary}
\integral_{B_{20R}} |\nabla w|^{\pmd} dx \leq C \integral_{B_{20R}}  |\nabla u|^{\pmd} dx,
\end{equation}
and for any ball $B_r(y)$ such that  $B_{7r}(y)\subset B_{20R}$

\begin{equation}\label{higher-boundary}
\bigg{(} \fintegral_{B_{r/2}(y)} |\nabla w|^{p+\tilde{\delta}_0} \, dx\bigg{)}^\frac{1}{p+\tilde{\delta}_0} \leq C \bigg{(} \fintegral_{B_{7r}(y)} |\nabla w|^{p-\tilde{\delta}_0 } \, dx \bigg{)}^{\frac{1}{p-\tilde{\delta}_0}}. 
\end{equation}
 Here the constants $C=C(n, p, \Lambda_0,\Lambda_1, \gamma)$.
\end{theorem}

Theorem \ref{higher-integrability-boundary}  was actually proved for a much larger class of domains and more general nonlinearities in \cite{AN}. More explicitly, the existence of $w$ and the bound \eqref{pmd-boundary} are contained in \cite[Corollary 3.5]{AN}; the higher integrability estimate \eqref{higher-boundary}
is obtained in \cite[Theorem 3.7]{AN}; and 
the comparison estimate is the result of \cite[Lemma 3.10]{AN}.

With $x_0 \in \partial\Om$ and $0 < R  < R_0/20$ as above, we now set $\rho = R(1-\ga)$. 
By Remark \eqref{domainremark},  there exists  a coordinate system $\{z_1,z_2,\ldots,z_n\}$ with the origin $0 \in \Om$ such that in this coordinate system $x_0 = (0,\ldots,0,-\rho \ga/(1-\ga)) \in \partial \Om$ and
\bea
\label{geometry_reifenberg}
B_{\rho}^+ (0) \subset \Om \cap B_{\rho} (0) \subset B_{\rho} (0) \cap \{(z_1,z_2,\ldots,z_n): z_n > -4\rho \ga\}.
\ena
Here recall that $B_{\rho}^+ (0) = B_{\rho}(0) \cap \{ (z_1, \dots,z_n): z_n > 0\}$ denotes an upper half ball in the corresponding coordinate system.

 With this $\rho$ and thanks to the existence and regularity of $w$ in Theorem \ref{higher-integrability-boundary},  we define another function $v\in w+ W_{0}^{1,\,p}(\Om_{\rho}(0))$
as the unique solution to the Dirichlet problem
\begin{equation}\label{vapprox}
\left\{ \begin{array}{rcl}
 \text{div}~ \overline{\mathcal{A}}_{B_{\rho}}(\nabla v)&=&0   \quad \text{in} ~\Om_{\rho}(0), \\ 
v&=&w  \quad \text{on}~ \partial \Om_{\rho}(0). 
\end{array}\right.
\end{equation}

We then set $v$ to be equal to $w$ in $\RR^n \setminus \omegarho (0)$.  The following boundary difference estimate can be proved 
in a way just similar to the proof of Lemma \ref{BMOapprox1}.

\begin{lemma}\label{BMOapprox2}
Under \eqref{monotone} and \eqref{ellipticity},  let $\de \in (0, \tilde{\de}_0)$, where $\tilde{\de}_0$ is  in Theorem 
\ref{higher-integrability-boundary}. Let $w$ and $v$ be as in \eqref{wapprox} and \eqref{vapprox}. 
For $\tau =\frac{p}{\tilde{\de}_0}\frac{(p+\tilde{\de}_0)}{ (p-1)}$, there exists a constant $C=C(n,p,\La_0,\La_1, \gamma)$ such that 
\begin{align*}
\fintegral_{B_{\rho}(0)} |\nabla v-\nabla w|^{\pmd}\, dx \leq C \left(\fintegral_{B_{\rho}(0)} \Upsilon(\aa,B_{\rho}(0))(x)^{\tau}\, dx\right)^{\min\{\pmd, \frac{\pmd}{p-1}\}/\tau} \times \\ \times \left(\fintegral_{B_{14\rho}(0)}|\nabla w|^{\pmd} \, dx\right) .
\end{align*}
%
%
\end{lemma}


As the boundary of $\Om$ can be very irregular, the $L^\infty$-norm of $\nabla v$ up to the boundary of $\Om$ could be unbounded. Therefore, we consider another equation: 
\begin{eqnarray}\label{perturbedpdeboundary}
\left\{ \begin{array}{rcl}
 \text{div}~ \overline{\mathcal{A}}_{B_{\rho}}(\nabla V)&=0 \quad \text{in} ~B^{+}_{\rho}(0), \\
V &= 0 \quad\text { on $T_{\rho}$},
\end{array}\right.
\end{eqnarray}
where $T_{\rho}$ is the flat portion of $\partial B^{+}_{\rho}(0)$. A function $V \in W^{1,p}(B_{\rho}^+ (0))$ is a weak solution of \eqref{perturbedpdeboundary} 
if its zero extension to $B_{\rho}(0)$ belongs to $W^{1,p} (B_{\rho} (0))$ and if 
\beas
\integral_{B_{\rho}^+ (0)} \overline{\mathcal{A}}_{B_{\rho}} (\nabla V) \cdot \nabla \phi \, dx = 0
\enas
for all $\phi \in W_0^{1,p} (B_{\rho}^+ (0)) $.

We shall need the following key perturbation result obtained earlier in \cite[Theorem 2.12]{NCP}. 

\begin{theorem}[\cite{NCP}] 
\label{leiberman_analog_boundary}
Suppose that $\aa$ satisfies \eqref{monotone} and \eqref{ellipticity}. For any $\ep>0$, there exists a small $\ga = \ga(n,p,\La_0,\La_1,\ep)>0$ such that if $v \in W^{1,p}(\omegarho (0))$ is a solutions of \eqref{vapprox} under the geometric setting \eqref{geometry_reifenberg}, then there exists a weak solution $V \in W^{1,p} (B_{\rho}^+ (0))$ of \eqref{perturbedpdeboundary} whose zero extension to $B_{\rho}(0)$ satisfies 
\beas
\| \nabla V \|^p_{L^{\infty} (B_{\rho/4}(0)) } \leq C \fintegral_{B_{\rho}(0)} |\nabla v|^p \, dx,
\enas
with $C = C(n,p,\La_0,\La_1)$ and
\beas
\fintegral_{B_{\rho/8}(0)} |\nabla v - \nabla V|^p \, dx \leq \ep^p \fintegral_{B_{\rho}(0)} |\nabla v|^p \, dx.
\enas

\end{theorem}

We now have the boundary analogue of Corollary \ref{mainlocalestimate-interior}. The proof of the following corollary follows with obvious modification as in \cite[Corollary 2.10]{MNP}. 
\begin{corollary}[\cite{MNP}] \label{mainlocalestimate-boundary} 
For any $\epsilon>0$ there exist    $\ga=\ga(n,p,\La_0,\La_1,\epsilon)>0$ and  $\tilde{\delta}_1 = \tilde{\delta}_1(n,p,\La_0,\La_1,\ep)\in (0, \tilde{\delta}_0)$,
where $\tilde{\delta}_0$ is as in Theorem \ref{higher-integrability-boundary},
such that the following holds with
 $\tau =\frac{p}{\tilde{\de}_0}\frac{(p+\tilde{\de}_0)}{ (p-1)}$.   If $\Om$ is  \gflt and if $u\in W_{0}^{1,\, \pmd}(\Om)$, $\de\in (0, \tilde{\de_1})$, is a very weak solution of 
\eqref{basicpde}  with
\begin{gather*}
\fintegral_{B_{20R}(x_0)}|\nabla u|^{\pmd}\chi_{\Om}\, dx \leq 1,  \fintegral_{B_{20R}(x_0)} |{\bf f}|^{\pmd} \chi_{\Om}\, dx \leq \ga^{\pmd},  
{\rm  ~and~} [\mathcal{A}]_\tau^{R_0} \leq \ga, 
\end{gather*}
where $x_0 \in\partial \Om$ and $R \in (0,R_0/20)$, 
then there is a function $$V\in W^{1,\, \infty}(B_{R/10}(x_0))$$ such that 
\beas
\norm{\nabla V}_{L^{\infty}(B_{R/10}(x_0))}\leq C_0 = C_0(n,p,\La_0,\La_1),
\enas
and
\bea \label{u-V}
\fintegral_{B_{R/10}(x_0)} | \nabla u- \nabla V|^{\pmd}\, dx \leq \epsilon^{\pmd} . 
\ena
\end{corollary}
\begin{proof}
 With  $x_0\in\partial\Om$ and  $R\in (0, R_0/20)$,  we  set
$\rho= R(1-\gamma)$. Also, extend both $u$ and ${\bf f}$ by zero to $\RR^n\setminus\Om$.
By Remark \eqref{domainremark}  and by  translating and  rotating if necessary, we may assume that $0\in\Om$,
$x_0=(0, \dots, 0, -\rho \gamma /(1-\gamma))$  and the geometric setting
\begin{equation}\label{BrhoOmrhosecond}
 B_{\rho}^{+}(0)\subset \Omega_{\rho}(0)\subset B_{\rho}(0)\cap \{ x_{n} > -4 \gamma\rho \}.
\end{equation}

Moreover, we shall further restrict  $\ga\in (0,1/45)$ so that we have $$B_{R/10}(x_0)\subset B_{\rho/8}(0).$$

We now choose $w$ and $v$ as in \eqref{wapprox} and \eqref{vapprox} corresponding to these $R$ and $\rho$.
Then, since $B_{14\rho}(0)\subset B_{20R}(x_0)$, there holds
\begin{equation*}
\fint_{B_{\rho}(0)} |\nabla v|^p dx \leq C  \fint_{B_{14\rho}(0)} |\nabla w|^p dx\leq C  \fint_{B_{20R}(x_0)} |\nabla u|^p dx \leq C.
\end{equation*}

By Theorem \ref{leiberman_analog_boundary} for any $\eta>0$ we can find a  $\ga=\ga(n, p, \Lambda_0, \Lambda_1, \eta)\in(0, 1/45)$ such that,
under \eqref{BrhoOmrhosecond}, there is a
function  $V\in W^{1,\, p}(B_\rho(0))\cap W^{1,\, \infty}(B_{\rho/4}(0))$ such that
$$
\norm{\nabla V}^{p}_{L^{\infty}(B_{R/10}(x_0))} \leq C\norm{\nabla V}^{p}_{L^{\infty}(B_{\rho/4}(0))} \leq C \fint_{B_{\rho}(0)} |\nabla v|^p dx \leq  C,
$$
and
\begin{equation*}
 \fint_{B_{\rho/8}(0)} | \nabla v- \nabla V|^pdx \leq \eta^p \fint_{B_{\rho}(0)} |\nabla v|^p dx \leq  C  \eta^p.
\end{equation*}

By H\"older's inequality, the last bound gives 
\begin{equation}\label{v-V}
\fint_{B_{\rho/8}(0)} | \nabla v- \nabla V|^{\pmd} dx \leq  C  \eta^{p-\delta}.
\end{equation}

Now writing 
\begin{equation*}
\fint_{B_{R/10}(x_0)} | \nabla u- \nabla V|^{p-\de}dx= \fint_{B_{\rho/8}(0)} | \nabla (u-w) +\nabla (w-v)+\nabla (v-V)|^{p-\de}dx,
\end{equation*}
and using \eqref{v-V} along with  Theorem \ref{higher-integrability-boundary} and Lemma  \ref{BMOapprox2},  
we obtain inequality \eqref{u-V} as desired (after choosing $\tilde{\delta}_1=\tilde{\delta}_1(\epsilon)$, $\eta=\eta(\epsilon)$, and  $\ga=\ga(\epsilon)$
 appropriately for any given $\epsilon>0$).
\end{proof}


\section{Weighted estimates}\label{sec7}

We now use Corollaries \ref{mainlocalestimate-interior} and \ref{mainlocalestimate-boundary}  to obtain the following technical result. 
\begin{proposition}\label{Byun-Wang-bdry}
Under \eqref{monotone}-\eqref{ellipticity}, there are  $\la = \la(n,p,\La_0,\La_1) > 1$ and $\tau = \tau(n,p,\La_0,\La_1)>1$ such that the following holds.
For any $\ep>0$, there exist  $\ga = \ga(n,p,\La_0,\La_1,\ep)>0$  and $\overline{\de} = \overline{\de}(n,p,\La_0,\La_1,\ep)>0$
such that if $u \in W_0^{1,\pmd}(\Om)$,  $\de\in(0,\overline{\de})$,  is a very weak solution to \eqref{basicpde} with $\Om$ being 
$(\ga, R_0)$-Reifenberg flat,  $[\mathcal{A}]_{\tau}^{R_0}\leq \ga,$
and if, for some ball $B_\rho(y)$ with $\rho < R_0/1200$,
\bea\label{BWassum2}
B_{\rho}(y) \cap \{x\in \RR^n&: \mm(|\nabla u|^{\pmd})^{\frac{1}{\pmd}}(x) \leq 1\}\cap \\
&\cap \{ x\in \RR^n: \mm(|{\bff}|^{\pmd}\chi_{\Om})^{\frac{1}{\pmd}} (x) \leq 
\ga \}\neq \emptyset,
\ena
then one has
\bea
\label{BWmain}
|\{ x\in \RR^n: \mm(|\nabla u|^{\pmd})^{\frac{1}{\pmd}} (x) > \lambda \} \cap B_{\rho}(y)|< \epsilon	 \, |B_{\rho}(y)|.
\ena
\end{proposition}
\begin{proof} 
By \eqref{BWassum2}, there exists an $x_0 \in B_{\rho}(y)$ such that for any $r > 0$, 
\bea
\label{BWassum3}
\fintegral_{B_{r}(x_{0})} |\nabla u|^{\pmd}\, dx \leq 1\quad \text{and}~~\fintegral_{B_{r}(x_{0})} \chi_{\Omega}|{\bff}|^{\pmd}\, dx \leq \ga^{\pmd}.
\ena

By the first inequality in \eqref{BWassum3},  for any $x \in B_{\rho}(y)$, there holds
\bea
\label{BW3}
\mm({|\nabla u|}^{\pmd})^{\frac{1}{\pmd}} (x) \leq \max \left\{ \mm(\chi_{B_{2\rho}(y)} {|\nabla u|}^{\pmd})^{\frac{1}{\pmd}} (x), \, 3^n \right\}.
\ena

To prove \eqref{BWmain}, it is enough to consider the case $B_{4\rho}(y) \subset \Om$ and the case $B_{4\rho}(y) \cap \partial \Om \neq \emptyset$. First we consider the latter.  Let
$y_0 \in B_{4\rho}(y) \cap \partial \Om$, we then have
\beas
B_{2\rho}(y) \subset B_{6\rho}(y_0) \subset B_{1200\rho}(y_0) \subset B_{1205 \rho}(x_0).
\enas

Thus by \eqref{BWassum3} we obtain
\beas
\fintegral_{B_{1200\rho}(y_{0})} |\nabla u|^{\pmd}\, dx \leq c\quad \text{and}~~\fintegral_{B_{1200\rho}(y_{0})} \chi_{\Omega}|{\bff}|^{\pmd}\, dx  \leq c\, \ga^{\pmd},
\enas
where $c=(1205/1200)^n$. Since $60\rho < R_0/20$, by Corollary \ref{mainlocalestimate-boundary} (with $R=60\rho$), there exists a $\tau(n,p,\Lambda_0,\Lambda_1) >1 $ such that the following holds. For any $\eta \in (0,1)$, there are $\ga = \ga(n,p,\La_0,\La_1,\eta)>0$, $\overline{\de} = \overline{\de}(n,p,\La_0,\La_1,\eta)>0$
such that if $\Om$ is a \gflt domain and $[\aa]_{\tau}^{R_0} \leq \ga$, then one can find a function $V \in W^{1,\infty} (B_{6\rho}(y_0))$ with
\bea
\label{BWassum4}
\| \nabla V\|_{L^{\infty} (B_{2\rho}(y))} \leq \| \nabla V\|_{L^{\infty} (B_{6\rho}(y_0))} \leq C_0,
\ena
and, for $\de\in (0,\overline{\de})$,
\bea
\label{BWassum5}
\fintegral_{B_{2\rho}(y)} |\nabla u - \nabla V|^{\pmd} \, dx \leq C \fintegral_{B_{6\rho}(y_0)} |\nabla u - \nabla V|^{\pmd}\, dx \leq C\, \eta^{\pmd}.
\ena

In view of \eqref{BW3}, \eqref{BWassum4} and the triangle inequality we see that, for $\la = \max \{ 3^n, 2C_0\}$, 
\beas
\{ x \in \RR^n & : \mm(|\nabla u|^{\pmd})^{\frac{1}{\pmd}} (x) > \la\} \cap B_{\rho} (y) \subset \\
&\subset \{ x \in \RR^n : \mm(\chi_{B_{2\rho}(y)}|\nabla u|^{\pmd})^{\frac{1}{\pmd}} (x) > \la\} \cap B_{\rho} (y) \\
&\subset \{ x \in \RR^n : \mm(\chi_{B_{2\rho}(y)}|\nabla u -\nabla V|^{\pmd})^{\frac{1}{\pmd}} (x) > \la/2\} \cap B_{\rho} (y).
\enas

Thus by the weak-type $(1,1)$ inequality for the Hardy-Littlewood maximal function and \eqref{BWassum5}, we find
\beas
|\{ x \in \RR^n & : \mm(|\nabla u|^{\pmd})^{\frac{1}{\pmd}} (x) > \la\} \cap B_{\rho} (y) | \leq \\
&\leq \frac{C}{\la^{\pmd}} \integral_{B_{2\rho}(y)} |\nabla u - \nabla V|^{\pmd} \, dx
\leq \frac{C}{C_0^{\pmd}} |B_{2\rho}(y)|\, \eta^{\pmd}.
\enas

This gives the estimate \eqref{BWmain} in the case $B_{4\rho}(y) \cap \partial \Om \neq \emptyset$, provided $\eta$ is appropriately chosen. The interior case $B_{4\rho}(y) \subset \Om$ can be obtained in a similar was by using Corollary \ref{mainlocalestimate-interior}, instead of Corollary \ref{mainlocalestimate-boundary}. 
\end{proof}

Proposition \ref{Byun-Wang-bdry} can now be used to obtain the following result which involves $\aw_{\infty}$ weights.

\begin{proposition}\label{contra}
Under \eqref{monotone}-\eqref{ellipticity}, there exist  $\la = \la(n,p,\La_0,\La_1)>1$ and $\tau = \tau(n,p,\La_0,\La_1)>1$ such that the following holds.
For any weight $w\in A_\infty$ and any $\epsilon>0$, there exist 
 $\ga = \ga(n,p,\La_0,\La_1,\ep,[w]_{\infty}) > 0$ and $\overline{\de}=\overline{\de}(n,p,\La_0,\La_1,\ep,[w]_{\infty}) > 0$ such that if $u \in W_0^{1,\pmd} (\Om)$, $\de\in (0,\overline{\de})$, is a very weak solution of \eqref{basicpde} with $\Om$ being $(\ga, R_0)$-Reifenberg flat, $[\aa]_{\tau}^{R_0} \leq \ga$, and if, for some ball $B_{\rho}(y)$ with $\rho < R_0/1200$, 
\beas
w(\{ x\in \RR^n: \mm(|\nabla u|^{\pmd})^{\frac{1}{\pmd}} (x) > \lambda \} \cap B_{\rho}(y))\geq \epsilon \, w(B_{\rho}(y)),
\enas
 then one has
\bea
\label{inclusion-1}
B_{\rho}(y) \subset \{x\in \RR^n&: \mm(|\nabla u|^{\pmd})^{\frac{1}{\pmd}} (x) > 1\} \, \cup  \\ 
& \cup\,  \{ x\in \RR^n: \mm(|{\bff}|^{\pmd}\chi_{\Om})^{\frac{1}{\pmd}}(x)> \ga\}.
\ena
\end{proposition}

\begin{proof}

Suppose that  $(\Xi_0,\Xi_1)$  is a pair of $\aw_{\infty}$ constants of $w$ and let $\la $ and $\tau$ be as in Proposition \ref{Byun-Wang-bdry}. Given $\ep > 0$, we choose a $\ga = \ga(\Xi_0,\Xi_1,\ep)$ and $\overline{\de}=\overline{\de}(\Xi_0,\Xi_1,\ep)$ as in Proposition \ref{Byun-Wang-bdry} with $[\ep/(2\Xi_0)]^{1/\Xi_1}$ replacing $\ep$. 
The proof then follows by a contradiction. To that end, suppose that the inclusion in \eqref{inclusion-1} fails for this $\ga$, then we must have that
\beas
B_{\rho}(y) \cap \{x\in \RR^n&: \mm(|\nabla u|^{\pmd})^{\frac{1}{\pmd}} (x) \leq 1\} \, \cap  \\ 
& \cap\,  \{ x\in \RR^n: \mm(|{\bff}|^{\pmd}\chi_{\Om})^{\frac{1}{\pmd}}(x)\leq \ga\} \neq \emptyset
\enas
for some $\delta\in (0, \overline{\de})$.
Hence by Proposition \ref{Byun-Wang-bdry}, if $\Om$ is a \gflt and $[\aa]_{\tau}^{R_0} \leq \ga$, there holds
\beas
|\{ x \in \RR^n: \, \mm(|\nabla u|^{\pmd})^{\frac{1}{\pmd}}(x) > \la \cap B_{\rho}(y)| \leq \left(\frac{\ep}{2\, \Xi_0}\right)^{1/\Xi_1} |B_{\rho}(y)|.
\enas

Thus using the $\aw_{\infty}$ characterization of $w$ (Lemma \ref{inversedoubling}), we immediately get  that
\beas
w(\{ x& \in \RR^n: \mm(|\nabla u|^{\pmd})^{\frac{1}{\pmd}} (x)  > \lambda \} \cap B_{\rho}(y)) \\
&\leq  \Xi_0 \left[\frac{ |\{ x \in \RR^n: \mm(|\nabla u|^{\pmd})^{\frac{1}{\pmd}}(x) > \la \} \cap B_{\rho}(y) |}{|B_{\rho}(y)|} \right]^{\Xi_1} w(B_{\rho}(y)) \\
&\leq \frac{\ep}{2} w(B_{\rho}(y)) < \ep\, w(B_{\rho}(y)).
\enas

This yields a contradiction and thus the proof is complete. 
\end{proof}

The following Calder\'on-Zygmund decomposition type lemma will allow us to iterate the result of  Proposition  \ref{contra} to 
obtain Theorem \ref{technicallemma} below. 
In the unweighted case various versions of this lemma have been obtained (see, e.g., \cite{CP, W, BW2}).
 The proof of this  weighted version was presented  in \cite{M-P1}. 

\begin{lemma} \label{vitalibdry}
 Let $\Omega$ be a \gflt  domain with $\ga<1/8$, and let $w$ be an $A_{\infty}$
 weight.
 Suppose that the sequence of balls $\{ B_{r}(y_i)\}_{i=1}^{L}$ with centers $y_i\in \overline{\Om}$
 and a common radius $r\leq R_0/4$  covers $\Om$.  Let  $C\subset D\subset \Omega$
 be measurable sets for which there exists $0<\epsilon < 1 $ such that
\begin{enumerate}
\item $ w(C)< \epsilon\, w(B_{r}(y_i)) $ for all $i=1,\dots, L$, and
\item for all $x \in \Omega$ and $\rho \in (0, 2r]$, if  $w(C\cap B_{\rho}(x))\geq \epsilon\,
w(B_{\rho}(x))$, then $B_{\rho}(x)\cap \Omega \subset D.$
\end{enumerate}
Then we have the estimate
 \begin{eqnarray*} w(C) \leq B \,\epsilon\, w(D)\end{eqnarray*}
for a constant $B$ depending only on $n$ and the $A_\infty$ constants of $w$. 
\end{lemma}

\begin{theorem}\label{technicallemma}
Under \eqref{monotone}-\eqref{ellipticity}, let $\la $ and $\tau$ be as in Proposition \ref{contra}.
Then for any weight $w\in A_\infty$ and any $\epsilon>0$, there exist constants $\ga=\ga(n,p,\La_0,\La_1,\epsilon, [w]_{\infty})>0$ and 
$\overline{\de}=\overline{\de}(n,p,\La_0,\La_1,\epsilon, [w]_{\infty})>0$ such that the following holds.
Suppose that $u \in W^{1, \, \pmd}_{0}(\Omega)$, $\de\in(0,\overline{\de})$, is a very weak solution of \eqref{basicpde}
in a \gflt domain $\Om$, with $[\aa]_{\tau}^{R_0} \leq \ga$. Suppose also that 
 $\{ B_{r}(y_i)\}_{i=1}^{L}$ is a sequence of balls with centers $y_i\in \overline{\Om}$ and a common  radius $0<r \leq R_0/4000$ that covers $\Omega$. If
for all $i=1,\dots, L$
\begin{equation}\label{hypo1bdry}
w(\{x\in \Omega: \mm(|\nabla u|^{\pmd})^{\frac{1}{\pmd}} (x) > \lambda \}) < \epsilon \, w(B_{r}(y_{i})),
\end{equation}
then for any $s>0$ and any integer $k\geq 1$ there holds
\beas
 w(\{x\in& \Omega: \mm(|\nabla u|^{\pmd})^{\frac{1}{\pmd}} (x) > \la^k \})^s \leq \\
&\leq \sum_{i=1}^{k} (A \ep)^{s i} \, w(\{ x\in \Omega : \mm(|{\bff}|^{\pmd}\chi_{\Om})^{\frac{1}{\pmd}}(x) > \ga \lambda^{(k-i)}\})^s + \\
&\qquad  + (A\ep)^{s k} \, w(\{x\in \Omega : \mm(|\nabla u|^{\pmd})^{\frac{1}{\pmd}}(x) > 1\})^s,
\enas
where the constant $A = A(n,[w]_{\infty})$. 
\end{theorem}

\begin{proof}
The theorem will be proved by induction on $k$. 
Given $w\in A_\infty$ and $\epsilon > 0$, we take $\ga=\ga(\ep, [w]_\infty)$ and $\overline{\de}=\overline{\de}(\epsilon, [w]_\infty)$ as in Proposition \ref{contra}. 
The case $k=1$ follows from 
Proposition \ref{contra} and Lemma  \ref{vitalibdry}. Indeed, for $\de\in (0,\overline{\de})$,  let 
\begin{gather*}
C  =  \{ x\in \Omega: \mm(|\nabla u|^{\pmd})^{\frac{1}{\pmd}} (x) > \lambda \} \\
D  =   \{x\in \Omega: \mm(|\nabla u|^{\pmd})^{\frac{1}{\pmd}}(x) > 1 \} \cup \{  x\in \Omega: \mm(|{\bff}|^{\pmd}\chi_{\Om})^{\frac{1}{\pmd}} (x) > \ga \}.
\end{gather*}

Then  from assumption (\ref{hypo1bdry}), it follows that
$w(C) < \epsilon \, w(B_{r}(y_{i}))$
for all $i=1,\dots, L$.  Moreover, if $y\in \Omega$ and $\rho \in (0, 2r)$ such that $w(C\cap B_{\rho}(y) ) \geq \epsilon\, w(B_{\rho}(y)) $,
then $0< \rho \leq R_0/1200$ and  $B_{\rho}(y) \cap \Omega\subset D$ by Proposition \ref{contra}. Thus  all hypotheses of Lemma \ref{vitalibdry} 
are satisfied, which yield, for a constant $B=B(n,[w]_{\infty})$, 
\beas
 w(C)^s & \leq B ^s\,  \ep^s \,  w(D)^{s} \\
& \leq B^s\, 2^s \ep^s \,   w(\{x\in \Omega: \mm(|\nabla u|^{\pmd})^{\frac{1}{\pmd}}(x)  > 1 \})^s + \\
 & \qquad + B^s\, 2^s \ep^s  \, w( \{ x\in \Omega: \mm(|{\bff}|^{\pmd}\chi_{\Om})^{\frac{1}{\pmd}} (x) > \ga \} )^s 
\enas
for any given $s>0$. This proves the case $k=1$ with $A=2 B$.  
Suppose now that the conclusion of the lemma is true for some $k>1$. Normalizing $u$ to
$u_{\lambda}  = u/\lambda$ and ${\bff}_{\lambda} = {\bff}/\lambda$, we see that for every $i=1, \dots, L$, 
\beas
w(\{ x\in \Omega&:  \mm(|\nabla u_{\lambda}|^{\pmd})^\fpmd (x) > \lambda \} ) = \\
&=  w(\{ x\in \Omega : \mm(|\nabla u|^{\pmd})^\fpmd (x)> \lambda^{2} \} )\\
&\leq w(\{ x\in \Omega : \mm(|\nabla u|^{\pmd})^\fpmd > \lambda \} )\\
 &< \epsilon \, w(B_{r}(y_{i})).
\enas
Here we  used the fact that $\lambda>1$ in the first inequality. Note that $u_\lambda$ solves 
\begin{eqnarray*}
\left\{ \begin{array}{rcl}
 \text{div}\,\tilde{\aa}(x, \nabla u_{\lambda})&=&\text{div}~|{\bf f}_{\lambda}|^{p-2}{\bf f}_{\lambda}  \quad \text{in} ~\Omega, \\
u&=&0  \quad \text{on}~ \partial \Omega,
\end{array} \right.
\end{eqnarray*}
where $\tilde{\aa}(x, \xi)={\aa}(x, \lambda \xi)/\lambda^{p-1}$ which obeys the same structural conditions \eqref{monotone}-\eqref{ellipticity}.
Thus by inductive hypothesis,  it follows that
\bea
\label{longchain}
w(\{ x&\in \Omega:  \mm(|\nabla u_{\lambda}|^{\pmd})^\fpmd (x) > \lambda^k \} )^s \\
&\leq \sum_{i=1}^{k} (A \ep)^{s i} \, w(\{ x\in \Omega : \mm(|{\bff}_\la|^{\pmd}\chi_{\Om})^{\frac{1}{\pmd}}(x) > \ga \lambda^{(k-i)}\})^s + \\
&\qquad  + (A\ep)^{s k} \, w(\{x\in \Omega : \mm(|\nabla u_\la|^{\pmd})^{\frac{1}{\pmd}}(x) > 1\})^s.
\ena
%
Finally,  applying the case $k = 1$ to the last term in \eqref{longchain} we conclude that
\beas
w(\{ x\in &\Omega : \mm(|\nabla u|^{\pmd})^\fpmd(x) > \lambda^{k+1} \})^s  \\
&\leq \sum_{i=1}^{k+1}(A \ep)^{s i} \, w(\{ x\in \Omega: \mm(|{\bff}|^{\pmd}\chi_{\Om})^\fpmd(x) > \ga \lambda^{k+1-i}\})^s\\
&\qquad  + \, (A \ep)^{s (k+1) } \, w(\{ x\in \Omega: \mm(|\nabla u|^{\pmd})^\fpmd(x) > 1 \})^s.
\enas
This completes the proof of the theorem.
\end{proof}

The following result is a characterization of functions in weighted Lorentz space and can easily be proved using methods in standard measure theory.
\begin{lemma}\label{distribution}
 Assume that $g\geq 0$ is a measurable function  in a bounded subset $\Om \subset \mathbb{R}^{n}$.
 Let $\theta > 0$, $\Lambda > 1$ be constants, and let
 $w$ be a weight in $\RR^n$.
 Then for $0< q, t< \infty$, we have
\begin{eqnarray*}
g \in L_w(q,\, t)(\Om) \Longleftrightarrow S := \sum_{k\geq 1} \Lambda^{t k}w(\{ x\in \Om: g(x) > \theta \Lambda^{k}\})^{\frac{t}{q}} < +\infty.
\end{eqnarray*}
Moreover, there exists a positive constant $C = C(\theta,\La,t)>0$ such that
\begin{eqnarray*}
C^{-1}\, S \leq \|g\|^{t}_{L_w(q,\, t)(\Om)} \leq C\, (w(\Om)^{\frac{t}{q}} + S).
\end{eqnarray*}
Analogously, for $0<q<\infty$ and $t=\infty$ we have
\begin{eqnarray*}
C^{-1} T\leq \norm{g}_{L_w(q,\, \infty)(\Om)}\leq C\, (w(\Om)^{\frac{1}{q}}+T) ,
\end{eqnarray*}
where $T$ is the quantity
\begin{eqnarray*}T:= \sup_{k\geq 1} \Lambda^k w(\{ x\in\Om: |g(x)|>\theta \Lambda^k\})^{\frac{1}{q}}.\end{eqnarray*}
\end{lemma}

We are now ready to obtain the main result of this section.

\begin{theorem}\label{main_theorem-M}
Suppose that $\aa$ satisfies \eqref{monotone}-\eqref{ellipticity}. Let $M>1$ and let $w$ be an $A_\infty$ weight. There exist  constants $\tau=\tau(n,p,\La_0, \La_1)>1$, $\de=\de(n,p,\La_0, \La_1, M, [w]_\infty)>0$ and $\ga=\ga(n,p,\La_0, \La_1, M, [w]_\infty)>0$  such that the following holds for
any $t\in (0,\infty]$ and $q\in (0, M]$. 
If $u \in W_0^{1,p-\de}(\Om)$ is a very weak solution of \eqref{basicpde} in a \gflt domain $\Om$ with $[\mathcal{A}]_\tau^{R_0}\leq \ga$,
then one has the  estimate 
\bea \label{globalbound-max}
\|\nabla u\|_{L_w(q,t)(\Om)} \leq C \|\mathcal{M}(|\bff|^{p-\de})^{\frac{1}{p-\de}} \|_{L_w(q,t)(\Om)},
\ena
where  the  constant $C = C(n,p,\La_0,\La_1, t, q, M, [w]_\infty, {\rm diam}(\Om)/R_0)$. 
\end{theorem}

\begin{remark}
{\rm The introduction of $M$ in the above theorem is just for a technical reason. It ensures that the constant $\delta$ is independent of $q$ as the proof of the 
theorem reveals. 
}
\end{remark}

\begin{remark}\label{upper-const}
{\rm It follows also from the proof of Theorem \ref{main_theorem-M}  that if $(\Xi_0, \Xi_1)$ is pair of $A_\infty$ constants of $w$
such that $\max\{ \Xi_0, 1/\Xi_1\}\leq \overline{\om}$ then the constants $\delta, \gamma$ and $C$ above can be chosen to depend just on the upper-bound $\overline{\om}$ instead of $(\Xi_0, \Xi_1)$.
} 
\end{remark}


\begin{proof}
Let  $\la = \la(n,p,\La_0,\La_1)$ and $\tau=\tau(n,p,\La_0,\La_1)$ be as in Theorem \ref{technicallemma}.
Take $\ep = \la^{-M} A^{-1} 2^{-1}$ and  choose $\delta= \overline{\delta}(n,p,\La_0,\La_1, \ep, [w]_{\infty})/2$,
where $A=A(n,[w]_\infty)$ and $\overline{\delta}$ are as in Theorem \ref{technicallemma}; thus 
$\delta= \delta(n,p,\La_0,\La_1, M, [w]_{\infty})$, which is independent of $q$.
Using Theorem \ref{technicallemma} we also get a 
constant $\ga = \ga(n,p,\La_0,\La_1, M, [w]_{\infty})>0$ for this choice of $\ep$.

We shall prove \eqref{globalbound-max} only for $t \in (0, \, \infty)$, as for $t= \infty$ the proof is just  similar. Choose a finite number of points $\{ y_i\}_{i=1}^L \subset \Om$ and a ball $B_0$ of radius $2 \diam (\Om)$ such that
\beas
\Om \subset \bigcup_{i=1}^L B_r(y_i) \subset B_0,
\enas
where $r = \min\{R_0/4000, \diam(\Om)\}$. We claim that we can choose $N$ large such that for $u_N= u/N$ and for all $i=1,\ldots,L$,
\bea
\label{smallness}
w(\{ x\in \Omega: \mm(|\nabla u_{N}|^{\pmd})^\fpmd(x) > \lambda) < \epsilon \, w(B_{r}(y_{i})).
\ena
 Indeed from the weak-type $(1,1)$ estimate for the maximal function, there exists a constant $C(n)>0$ such that 
\beas
|\{x \in \Omega: \mm(|\nabla u_{N}|^{\pmd})^\fpmd(x) > \lambda\}| < \frac{C(n)}{(\lambda N)^{\pmd}} \integral_{\Omega}|\nabla u|^{\pmd}\, dx. 
\enas

If $(\Xi_0,\Xi_1)$  is a pair of $\aw_{\infty}$ constants of $w$, then using Lemma \ref{inversedoubling}, we see that 
\bea
\label{equa1}
w(\{x \in \Omega:& \mm(|\nabla u_{N}|^{\pmd})^\fpmd(x) > \lambda\})  \\
& < \Xi_0 \left(\frac{C(n)}{(\lambda N)^{\pmd}|B_0|} \integral_{\Omega}|\nabla u|^{\pmd}\, dx. \right)^{\Xi_1} w(B_0).
\ena

Also, there are $C_1=C_1(n, [w]_\infty)\geq 1$ and $p_1=p_1(n, [w]_\infty)\geq 1$ such that
\bea
\label{equa2}
w(B_0) \leq C_1 \left( \frac{|B_0|}{|B_r(y_i)|} \right)^{p_1} w(B_r(y_i))
\ena
for every $i=1,2,\dots, L$. This follows from the so-called {\it strong doubling property} of $A_\infty$ weights (see, e.g., \cite[Chapter 9]{Gra}). 
In view of  \eqref{equa1} and \eqref{equa2}, we now choose $N$ such that
\beas
\frac{C(n)}{(\lambda N)^{\pmd}|B_0| } \integral_{\Omega}|\nabla u|^{\pmd}\, dx &= \left(\frac{|B_r(y_i)|}{|B_0|}\right)^{p_1/\Xi_1} \left( \frac{\ep}{\Xi_0 C_1} \right)^{1/\Xi_1}. 
\enas

This gives the desired  estimate \eqref{smallness}. Note that for this $N$ we have 
\bea\label{N-est}
N  &\leq  C |B_0|^{\frac{-1}{\pmd}} \|\nabla u\|_{L^{\pmd}(\Om)} \leq  C  |B_0|^{\frac{-1}{\pmd}}  \|\bff \chi_{\Om} \|_{L^{\pmd}(B_0)}\\
&\leq  C    \mathcal{M}(|\bff|^{\pmd} \chi_{\Om})(x)^{\frac{1}{\pmd}}
\ena
for all $x\in\Om$. Here  $C = C(n,p,\La_0,\La_1, M,  [w]_{\infty}, \diam(\Om)/R_0)$ and the second inequality follows from Theorem \cite[Theorem 1.2]{AN}.

With this $N$, we denote by $$S= \sum_{k = 1}^{\infty}{\lambda}^{t k}w(\{ x\in \Om: \mm(|\nabla u_{N}|^{\pmd})^\fpmd (x) > {\lambda}^{k} \})^{\frac{t}{q}}$$
and for $J\geq 1$ let 
\beas
S_J= \sum_{k = 1}^{J}{\lambda}^{t k}w(\{ x\in \Om: \mm(|\nabla u_{N}|^{\pmd})^\fpmd (x) > {\lambda}^{k} \})^{\frac{t}{q}}
\enas
be its partial sum.  By Lemma \ref{distribution}, we see that 
\bea
\label{equa3}
C^{-1} S \leq \| \mm(|\nabla u_N|^{\pmd})^{\fpmd} \|_{L_w(q,t)(\Om)}^t \leq C(w(\Om)^{\frac{t}{q}} + S).  
\ena

By \eqref{smallness} and Theorem \ref{technicallemma}, we find
\beas
S_J \leq & \sum_{k = 1}^{J}{\la}^{tk} \left[ \sum_{j=1}^{k} (A \ep)^{\frac{t}{q}j} w(\{ x\in \Om: \mm(|{\bff}_{N}|^{\pmd}\chi_{\Om})^\fpmd(x) > \ga \la^{(k-j)}\})^{\frac{t}{q}} \right]\\
& \quad + \sum_{k = 1}^{J} \la^{tk}  (A \ep)^{\frac{t}{q}k}  w(\{x\in \Om: \mm(|\nabla u_{N}|^{\pmd})^\fpmd(x) > 1\})^{\frac{t}{q}}. \\
\enas
Here recall that $\ep = \la^{-M} A^{-1} 2^{-1}$ ans  $A = A(n, [w]_{\infty})$.
Now interchanging the order of summation,  we  get
\beas
S_J & \leq  \sum_{j = 1}^{J} (A \ep \la^q)^{\frac{t}{q} j}  \left[ \sum_{k=j}^{J} \la^{t(k-j)}  w(\Om\cap \{\mm(|{\bff}_{N}|^{\pmd}\chi_{\Om})^\fpmd > \ga \la^{(k-j)}\})^{\frac{t}{q}} \right]\\
& \quad + \sum_{k = 1}^{J} (A \ep \la^{q})^{\frac{t}{q}k}    w(\{x\in \Om: \mm(|\nabla u_{N}|^{\pmd})^\fpmd(x) > 1\})^{\frac{t}{q}} \\
& \leq  C \left[ \| \mm(|\bff_N|^{\pmd})^\fpmd\|_{L_w(q,t)(\Om)}^t + w(\Om)^{\frac{t}{q}}    \right] \sum_{j=1}^{\infty} 2^{-\frac{t}{q}j} \\
& \leq  C \left[ \| \mm(|\bff_N|^{\pmd})^\fpmd\|_{L_w(q,t)(\Om)}^t + w(\Om)^{\frac{t}{q}}    \right] 
\enas
for a  constant $C = C(n,p,\La_0,\La_1, q, t, M, [w]_{\infty})$.  Letting $J \rightarrow \infty$  and   making use of \eqref{equa3}, we arrive at
\beas
 \| \mm(|\nabla u_N|^{\pmd})^\fpmd\|_{L_w(q,t)(\Om)}^t  \leq C  \left[ \| \mm(|\bff_N|^{\pmd})^\fpmd\|_{L_w(q,t)(\Om)}^t + w(\Om)^{\frac{t}{q}}   \right] .
\enas

This gives
\beas
 \| \nabla u\|_{L_w(q,t)(\Om)}  \leq C  \left[ \| \mm(|\bff|^{\pmd}\chi_\Om)^\fpmd\|_{L_w(q,t)(\Om)} + N w(\Om)^{\frac{1}{q}}    \right],
\enas
which in view of \eqref{N-est} yields the desired estimate.
\end{proof}

\begin{appendices}
\section{Appendix: Proof of Theorem \ref{theorem_extrapolation}}\label{appendix-A}

In this appendix, we provide a complete proof of Thereom \ref{theorem_extrapolation}.

\begin{proof}
First we consider the sub-natural case $p-1<q<p$. To that end, let $w \in \awpq{q}{p-1}$ and suppose that $\bff \in L^p(\Om,\RR^n) \cap L^q_w(\Om,\RR^n)$ satisfying 
\eqref{weight11} for all $v \in \awpq{p}{p-1}$. Extend both $\bff$ and $u$ by zero to $\RR^n \setminus \Om$ and define 
$$\mr(\bff)(x) := \sum_{k=0}^{\infty} \frac{\mm^{(k)}(|\bff|^{p-1})(x)}{2^k \|\mm\|^k_{L^{q/(p-1)}_w \rightarrow L^{q/(p-1)}_w}}.$$ 
Here $\mm^{(k)}=\mm\circ\mm\circ\cdots\circ\mm$ ($k$ times) and  note that (see, e.g., \cite[Chapter 9]{Gra}) 
\bea \label{ww10}\|\mm\|_{L^{q/(p-1)}_w \rightarrow L^{q/(p-1)}_w } \leq C(n,p,q, [w]_{\frac{q}{p-1}}).\ena 

Now it is easy to  observe from the definition of $\mr(\bff)$ that 
\bea \label{eq111} |\bff(x)|^{p-1} \leq \mr(\bff)(x) , \quad \textrm{and} \quad \|\mr(\bff) \|_{L^{q/(p-1)}_w}\leq 2 \|\bff \|_{L^q_w}^{p-1}. \ena 

An important result which we shall need is the following estimate:
\bea \label{eq112} \mr(\bff)^{-\frac{(p-q)}{(p-1)}}\,  w \in \awpq{p}{p-1}\quad   \textrm{with} \quad [\mr(\bff)^{-\frac{(p-q)}{(p-1)}}\,  w ]_{\frac{p}{p-1}} \leq C([w]_{\frac{q}{p-1}})	.\ena
The proof of \eqref{eq112} is obtained as follows:  it follows from \eqref{ww10} and the definition of $\mr(\bff)$
that $$\mm(\mr(\bff)) \leq C([w]_{\frac{q}{p-1}}) \mr(\bff),$$ and thus
we get  that 
$$\mr(\bff)(x)^{-1} \leq C([w]_{\frac{q}{p-1}}) \left( \frac{1}{|B|} \integral_B \mr(\bff) \, dy \right)^{-1}$$
for any ball $B\subset\RR^n$ containing $x$.
Set now $s=\frac{(p-q)}{(p-1)} \frac{q}{p}$. Using the last inequality, we find for any ball $B\subset\RR^n$,
\bea
\label{ww1}
\fintegral_B \mr(\bff)^{-s \frac{p}{q}} \, w \, dx \leq C([w]_{\frac{q}{p-1}}) \left( \fintegral_B \mr(\bff) \, dy \right)^{-s\frac{p}{q}} \left( \fintegral_B w(x)\, dx \right).
\ena

On the other hand, by H\"older's inequality  there holds
\bea
\label{ww2}
&\left( \fintegral_B [\mr(\bff)^{-s \frac{p}{q}} w(x)]^{1-p}\, dx \right)^{\frac{1}{p-1}}  = \left( \fintegral_B \mr(\bff)^{p-q} w(x)^{1-p}\, dx \right)^{\frac{1}{p-1}} \\
& \qquad \qquad \qquad \qquad  \leq \left( \fintegral_B \mr(\bff) \, dx \right)^{\frac{p-q}{p-1}} \left( \fintegral_B w(x)^{\frac{1-p}{1-p+q}} \, dx \right)^{\frac{1-p+q}{p-1}}.
\ena

Multiplying \eqref{ww1} by \eqref{ww2}, we obtain the conclusion stated in \eqref{eq112}.

We now obtain by H\"older's inequality
\bea
\label{ww7}
&\integral_{\RR^n} |\nabla u|^q w \, dx = \integral_{\RR^n} |\nabla u|^q \, \mr(\bff)^{-s} \mr(\bff)^{s} w \, dx \\
& \leq \left( \integral_{\RR^n} |\nabla u|^p \mr(\bff)^{-s.\frac{p}{q}}  w \, dx\right)^{q/p}   \left( \integral_{\RR^n} \mr(\bff)^{s.\frac{q}{p-q}} w \, dx \right)^{(p-q)/p}.
\ena

By making use of the hypothesis of the theorem along with \eqref{eq111}, we can then estimate the right hand side of \eqref{ww7} as
\beas
\integral_{\RR^n} |\nabla u|^q w \, dx& \leq C\left( [\mr(\bff)^{-s.\frac{p}{q}} w]_{\frac{p}{p-1}}  \right) \left( \integral_{\RR^n} |\bff|^p\,  \mr(\bff)^{-s.\frac{p}{q}}  w \, dx\right)^{q/p} \times \\
& \qquad \qquad \qquad \times  \left( \integral_{\RR^n} \mr(\bff)^{s.\frac{q}{(p-q)}} w \, dx \right)^{(p-q)/p}\\
& \leq C\left( [\mr(\bff)^{-s.\frac{p}{q}} w]_{\frac{p}{p-1}}  \right) \left( \integral_{\RR^n} \mr(\bff)^{\frac{q}{p-1}} w \, dx \right)\\ 
& \leq C\left( [\mr(\bff)^{-s.\frac{p}{q}} w]_{\frac{p}{p-1}}  \right) 2^{\frac{q}{p-1}} \|\bff\|_{L^q_w}^q. 
\enas

Then applying \eqref{eq112}, we obtain  \eqref{extrapolation_result} in the case $p-1<q<p$. 

We now consider the case $p<q<\infty$ and in this regard, we fix a $w \in \awpq{q}{p-1}$ and let $\bff \in L^p(\Om,\RR^n) \cap L^q_w(\Om,\RR^n)$ be as in the  theorem. For any $h \in L^{(q/p)'}_w(\RR^n)$, define $$\mr'(h)(x) := \sum_{k=0}^{\infty} \frac{(\mm')^{(k)}(|h|^{\frac{\left(q/p\right)'}{\left(q/(p-1)\right)'}})(x)}{2^k \|\mm'\|^k_{L^{(q/(p-1))'}_w\rightarrow L^{(q/(p-1))'}_w}},$$
where $\mm'(h) := \frac{\mm(hw)}{w}$ and $(q/p)' = \frac{q}{q-p}$, $(q/(p-1))'=\frac{q}{q-p+1}$ denote the conjugate H\"older exponents. 
Then it is easy to  observe that 
\bea \label{eq113} |h|^{\frac{\left(q/p\right)'}{\left(q/(p-1)\right)'}} (x) \leq \mr'(h)(x) , \  \textrm{and} \ \|\mr'(h) \|_{L^{\left(q/(p-1)\right)'}_w}\leq 2 \|h \|_{L^{(q/p)'}_w}^{\frac{\left(q/p\right)'}{\left(q/(p-1)\right)'}}. \ena

We now choose an $h \in L^{(q/p)'}_w(\RR^n)$ with $\|h\|_{L^{(q/p)')}_w} = 1$ such that 
\bea
\label{eq115} 
\integral_{\RR^n} |\nabla u|^q \, w(x) \, dx = \| |\nabla u|^p \|_{L^{q/p}_w}^{q/p}= \left( \integral_{\RR^n} |\nabla u|^p h(x)\, w(x) \, dx \right)^{q/p}.
\ena

For this choice of $h$, define $H := \left[ \mr'(h) \right]^{\frac{\left(q/(p-1)\right)'}{\left(q/p\right)'}}$. It is easy to see from \eqref{eq113}
that $0 \leq h \leq H$. We now prove the following important estimate:
\bea
\label{ww3}
(Hw) \in \aw_{\frac{p}{p-1}} \quad   \textrm{with} \quad [Hw]_{\frac{p}{p-1}} \leq C([w]_{\frac{q}{p-1}}). 
\ena

Analogous to \eqref{ww10}, we observe that $\mm'(\mr'(h)) \leq C([w]_{\frac{q}{p-1}}) \mr'(h).$ Thus for any ball $B$ containing $x$, 
\beas
(Hw)(x)^{1-p} \leq C([w]_{\frac{q}{p-1}}) \left( \fintegral_B H^{\frac{(q/p)'}{(q/(p-1))'}} w(y) \ dy \right)^{\frac{(q/(p-1))'}{(q/p)'} (1-p)} w(x)^{\frac{1-p}{q-p+1}},
\enas
where we have used the fact that $\left(\frac{(q/(p-1))'}{(q/p)'} -1\right) (p-1) = \frac{1-p}{q-p+1}$.
With this we obtain the estimate
\bea\label{ww4}
&\left( \fintegral_B (Hw)^{1-p} \ dx \right)^{\frac{1}{p-1}}\\
 &\leq C([w]_{\frac{q}{p-1}}) \left( \fintegral_B H^{\frac{(q/p)'}{(q/(p-1))'}} w \ dy \right)^{-\frac{(q/(p-1))'}{(q/p)'} }  \left( \fintegral_B w ^{\frac{1-p}{q-p+1}} \ dx \right)^{\frac{1}{p-1}} 
\ena
for all balls $B\subset\RR^n$.

On the other hand, by H\"older's inequality, we obtain
\bea
\label{ww5}
\fintegral_Q H w \ dx \leq \left( \fintegral_Q H^{\frac{(q/p)'}{(q/(p-1))'}} w \ dx \right)^{\frac{(q/(p-1))'}{(q/p)'}} \left( \fintegral_Q w \ dx \right)^{1-\frac{(q/(p-1))'}{(q/p)'}}.
\ena

Multiplying \eqref{ww4} by \eqref{ww5} and observing that $1-\frac{(q/(p-1))'}{(q/p)'} = \frac{1}{q-p+1}$, we get
\beas
\relax [Hw]_{\frac{p}{p-1}} &\leq C([w]_{\frac{q}{p-1}}) \left( \fintegral_Q w^{\frac{1-p}{q-p+1}} \ dx \right)^{\frac{1}{p-1}}\left( \fintegral_Q w \ dx \right)^{\frac{1}{q-p+1}} \leq C([w]_{\frac{q}{p-1}}),
\enas
which completes the proof of \eqref{ww3}.

Using our hypothesis on ${\bf f}$ and H\"older's inequality  we now obtain
\bea
\label{eq116}
&\integral_{\RR^n} |\nabla u|^p \, h\, w \, dx  \leq \integral_{\RR^n} |\nabla u|^p \, H\, w \, dx \\
& \leq C\left([Hw]_{\frac{p}{p-1}} \right) \integral_{\RR^n} |\bff|^p \, H\, w \, dx \\
&  \leq C\left([Hw]_{\frac{p}{p-1}}\right) \left( \integral_{\RR^n} |\bff|^q \, w \, dx \right)^{p/q}  \left(  \integral_{\RR^n} |H|^{(q/p)'} \, w \, dx \right)^{1/(q/p)'}.
\ena

Concerning the last term on the right, we have 
\bea
\label{eq117}
\integral_{\RR^n} |H|^{(q/p)'} \, w \, dx  & = \integral_{\RR^n} \mr'(h)^{(q/(p-1))'} \, w \, dx\\
& = \| \mr'(h) \|_{L^{(q/(p-1))'}_w}^{(q/(p-1))'}  \leq 2^{(q/(p-1))'} \|h\|_{L^{(q/p)'}_w}^{(q/p)'}, 
\ena
where the last inequality follows from \eqref{eq113}. 

Substituting \eqref{eq117} into \eqref{eq116} and recalling   \eqref{eq115}, we obtain the desired estimate when $p<q<\infty$.
\end{proof} 
\end{appendices}

\end{document}